\documentclass[]{interact}
\usepackage{epstopdf}
\usepackage[caption=false]{subfig}
\usepackage[numbers,sort&compress]{natbib}
\usepackage{comment}
\usepackage{color,xcolor}
\usepackage{amsmath}
\bibpunct[,]{[}{]}{,}{n}{,}{,}
\renewcommand\bibfont{\fontsize{10}{12}\selectfont}
\makeatletter
\def\NAT@def@citea{\def\@citea{\NAT@separator}}
\makeatother

\theoremstyle{plain}
\newtheorem{theorem}{Theorem}[section]
\newtheorem{lemma}[theorem]{Lemma}
\newtheorem{corollary}[theorem]{Corollary}

\theoremstyle{definition}
\newtheorem{definition}[theorem]{Definition}
\newtheorem{example}[theorem]{Example}

\theoremstyle{remark}

\begin{document}

\title{The Bott-Duffin drazin inverse and its application}

\author{
\name{Lu Zheng\textsuperscript{a}\thanks{CONTACT Xiangyu Zhang. Email: xiangyuz021@163.com}, Xiangyu Zhang\textsuperscript{b,$\ast$}, Kezheng Zuo\textsuperscript{a} and Jing Zhou\textsuperscript{a}}
\affil{\textsuperscript{a} Teaching Section of Mathematics, Hubei Engineering Institute, Huangshi, China;}
 \affil{\textsuperscript{b} Jiaozhou Vocational Education Center School, Qingdao, China}}
\maketitle

\begin{abstract}
  The paper introduce a new type of generalized inverse, called Bott-Duffin drazin inverse  (or,
  in short, BDD-inverse) of a complex square matrix, and give some of its
  properties, characterizations and representations. Furthermore,
  We discuss the problem of the minimum P-norm solution of the constraint matrix equation by using the Bott-Duffin drazin inverse, and give Cramer's rule for this minimum P-norm solution.
\end{abstract}


\section{Introduction}
The notation $\mathbb{C}^{m\times{n}}$ represents the set of $m\times{n}$ complex matrices.
$L\leq\mathbb{C}^{n}$ means $L$ is a subspace of $\mathbb{C}^{n}$.
$\textup{rank}{(A)}$ is the rank of $A\in\mathbb{C}^{m\times{n}}$.
The smallest nonnegative integer $k$ for which $\textup{rank}{({A}^{k})}=\textup{rank}{({A}^{k+1})}$ is called the index of $A\in\mathbb{C}^{n\times{n}}$ and is denoted by $\textup{Ind}{(A)}$.
Bott and Duffin {\cite{bott1953algebra}} introduced the constrained inverse of the square matrix in electric network theory, namely the Bott-Duffin inverse.
When $A\in\mathbb{C}^{n\times{n}}$, $L\leq\mathbb{C}^{n}$ and $AP_{L}+P_{L^{\perp}}$ is invertible,
the Bott-Duffin inverse of $A$ with respect to $L$, denoted by $A_{(L)}^{(-1)}$, is defined by $A_{(L)}^{(-1)}=P_{L}{(AP_{L}+P_{L^{\perp}})}^{-1}$.
After this, many authors conducted further research on the Bott-Duffin inverse ({see}\cite{ben1963generalized,chen1990g,wei2003condition,pian2005algebraic,hughes2014minimality,tempo2018emerging,wu2023new}).
\par  It is interesting to mention that when $AP_{L}+ P_{L^{\perp}}$ is singular, Chen {\cite{yonglin1990generalized}} naturally defines the generalized Bott-Duffin inverse of $A\in\mathbb{C}^{n\times{n}}$ (denoted by $A_{(L)}^{(\dagger)}$) and extends the Bott-Duffin inverse to the general case.
Inspired by the paper of Chen {\cite{yonglin1990generalized}}, it is worth considering that if $AP_{L}+P_{L^{\perp}}$ is drazin invertible, can a new generalized inverse be obtained?
Therefore, our main aim is to introduce and investigate a new generalized
inverse, namely the Bott-Duffin drazin inverse.
\begin{definition}\label{D1}
  Let $A\in\mathbb{C}^{n\times{n}}$, $L\leq\mathbb{C}^{n}$. If $\textup{Ind}(AP_L + P_{L^\perp})=k$, then 
  \begin{equation}\label{E1}
    A_{(L)}^{(D)}=P_{L}{(AP_{L}+P_{L^{\perp}})}^{D},
  \end{equation}
  is called the Bott-Duffin drazin inverse of $A$ with respect to $L$.
\end{definition}
\par Let us now recall notions of several generalized inverses and notations.
\par As usual, $\mathbb{C}_{r}^{m\times{n}}$ stands for the set of all ${m\times{n}}$
complex matrices of rank $r$ while $I_{n}$ represents the identity matrix in $\mathbb{C}^{n\times{n}}$.
For $A\in\mathbb{C}^{m\times{n}}$,  the symbols $A^{\ast}$, $A^{-1}$, $\mathcal{R}{(A)}$ and $\mathcal{N}{(A)}$
stand for the conjugate transpose, the inverse ${(m=n)}$, the range space and the null space of $A$.
\par  The matrix formed by replacing column $i$ of $A$ with $b$ is denoted by $A{(i\rightarrow b)}$, where $A\in\mathbb{C}^{m\times{n}}$
and $b\in\mathbb{C}^{m}$.
Moreover, $O$ will refer to  the null matrix. $L^{\perp}$ is the orthogonal complement subspace of $L$. The dimension of $L$ is denoted by $\dim(L)$.
For two subspaces $S$ and $T$ such that $S\oplus T=\mathbb{C}^{n}$,
$P_{S,T}$ stands for the oblique projector onto $S$ along $T$.
The orthogonal projector onto a subspace $S$ will be indicated by $P_{S}$.
A matrix $X\in\mathbb{C}^{n\times{m}}$ that satisfies the equality $XAX=X$ is called the $\{2\}$-inverse of $A$.
\par  The Moore-Penrose inverse of $A\in\mathbb{C}^{m\times{n}}$ is the unique matrix ${A^{\dagger}}\in\mathbb{C}^{n\times{m}}$
satisfying the following four equations (see {\cite{penrose1955generalized,ben2003generalized,wang2018generalized}}):
\[AA^{\dagger}A=A,~A^{\dagger}AA^{\dagger}=A^{\dagger},~{(AA^{\dagger})}^{\ast}=AA^{\dagger}~\text{and}~{(A^{\dagger}A)}^{\ast}=A^{\dagger}A.\]
\par We recall that the Drazin inverse of $A\in\mathbb{C}^{n\times{n}}$
is the unique matrix $A^{D}\in\mathbb{C}^{n\times{n}}$ satisfying
the following three equations (see {\cite{ben2003generalized}}):
\begin{equation}\label{dra}
A^{D}AA^{D}=A^{D},~AA^{D}=A^{D}A~\text{and}~A^{D}A^{k+1}=A^{k},
\end{equation}
where $k=\textup{Ind}{(A)}$. We use the symbol $A^{\pi}$ represents $I_{n}-AA^{D}$, where $A\in\mathbb{C}^{n\times{n}}$.
\par If $X\in\mathbb{C}^{n\times{m}}$ satisfies the following three equations:
\[XAX=X,~\mathcal{R}{(X)}=S~\text{and}~\mathcal{N}{(X)}=T,\]
where $S,T\leq\mathbb{C}^{n}$, then $X$ is unique and is defined by $A_{S,T}^{(2)}$ (see {\cite{ben2003generalized,wang2018generalized}}).
\vspace{0.5em}
\par The main contributions of this paper are summarized as follows:
\begin{enumerate}
\item We give the equivalent conditions of $\textup{Ind}{(AP_{L}+P_{L^{\perp}})}=k$ and some properties of the BDD-inverse.
In particular, it shows that BDD-inverse is an outer inverse with prescribed range and null space.
\item Some characterizations of the BDD-inverse are obtained by using range space, null space, matrix equation, projectors.
\item The BDD-inverse is characterized by the rank equation, and some representations of the BDD-inverse are given.
In addition, the representation of the BDD-inverse is given by restricting $P_{L}AP_{L}$ in a specific subspace.
\item We use the BDD-inverse to study the problem of the minimum P-norm solution of the constrained matrix equations. Furthermore, we give Cramer's rule for the minimum P-norm solution.
\end{enumerate}

\par The rest of this paper is organized as follows: In Section 2, several  auxiliary lemmas are given.
In Section 3, we derive some equivalence conditions of $\textup{Ind}{(AP_{L}+P_{L^{\perp}})}$, and several properties of BDD-inverse of square matrix are given.
Section 4 provides some characteristics of BDD-inverse of square matrix.
Section 5 offers some representations of BDD-inverse of square matrix.
Finally, Section 6 is devoted to the study of the minimum P-norm solution of the constraint equation, and gives the Cramer's rule for its solution.
\section{Premliminaries}
\noindent First, we introduce some basic lemmas.
\begin{lemma}\label{tri}\emph{\cite{zhang2023exact}}
Let $M=\begin{bmatrix}A & B \\O & E\end{bmatrix}$
and $N=\begin{bmatrix}E & O \\ B & A\end{bmatrix}\in\mathbb{C}^{n\times{n}}$,
where $A$ and $E$ are square matrices such that $r=\operatorname{ind}(A)$
and $s=\operatorname{ind}(E)$. Then
\[M^{D}=\begin{bmatrix}
  A^{D} & X \\
  O & E^{D}
\end{bmatrix}~\text{and}~
  N^{D}=\begin{bmatrix}
  E^{D} & O \\
  X & A^{D}
  \end{bmatrix},\]  
where
  \[X=\sum_{i=0}^{s-1} A^{(i+2) D} B E^{i}E^{\pi}+A^{\pi}\sum_{i=0}^{r-1} A^{i} B E^{(i+2) D}-A^{D} B E^{D}.\]
\end{lemma}

\begin{lemma}\label{aco}\emph{\cite{bu2013some}}
  Let $M=\begin{bmatrix}
    A & O\\
    C & O
    \end{bmatrix}\in\mathbb{C}^{n \times n}, A~\text {is a square matrix, then}$
    \[M^{D}=\begin{bmatrix}
    A^{D} & O \\
    C{(A^{D})}^{2} & O
    \end{bmatrix}.\]
\end{lemma}


\begin{lemma}\label{pq0}\emph{\cite{hartwig2001some}}
Let $M,~N\in\mathbb{C}^{n\times{n}}$ and $MN=NM=O$. Then
\[{(M+N)}^{D}=M^{D}+N^{D}.\]
\end{lemma}

\begin{lemma}\label{ccc}\emph{\cite[Theorem 7.8.4]{campbell2009generalized}}
Let $M,~N\in\mathbb{C}^{n\times{n}}$. Then
\[{(M N)}^{D}=M{({(N M)}^{2})}^{D} N .\]
\end{lemma}

Let $A\in\mathbb{C}^{n\times n}$ and $L\leq \mathbb{C}^{n}$.
To discuss some properties and characteristics of the BDD-inverse, we will consider a suitable matrix factorization of $A$ with respect to $L$.
Since there exists a unitary matrix $U \in \mathbb{C}^{n\times n}$ such that
  \begin{equation}\label{E2}
  P_L = U \begin{bmatrix}
    I_l & O \\
    O & O
    \end{bmatrix} U^*,
  \end{equation}
    where $ l = \dim(L)$. A matrix $A$ can be written as
    \begin{equation}\label{E3}
    A = U \begin{bmatrix} 
      A_L & B_L \\
      C_L& E_L
      \end{bmatrix} U^*,
    \end{equation}
    where $A_L \in \mathbb{C}^{l \times l}$, $B_L \in \mathbb{C}^{l \times (n-l)}$, $C_L \in \mathbb{C}^{(n-l) \times l}$and $E_L \in \mathbb{C}^{(n-l) \times (n-l)}$.

\noindent In the following, we provide the matrix block representation of $A_{(L)}^{(D)}$.

\begin{lemma}\label{ler}
    Let $P_{L}$ and $A$ be given by $(\ref*{E2})$ and $(\ref*{E3})$, respectively.
    \begin{equation}\label{E4}
    A^{(D)}_{(L)}=U\begin{bmatrix}{A_{L}}^{D}&O\\O&O \end{bmatrix}U^*.
    \end{equation}
    \textbf{\textit{Proof.}}
    \textup{Using $(\ref*{E2})$ and $(\ref*{E3})$, we get
\begin{equation}\label{E5}
AP_{L}+P_{L^\perp} = U \begin{bmatrix} 
  A_L & O \\
  C_L& I_{n-l}
  \end{bmatrix} U^*.
\end{equation}
It follows from $(\ref*{E5})$ and Lemma $\ref*{tri}$ that
\[{(AP_{L}+P_{L^\perp})}^{D}=U\begin{bmatrix}{A_{L}}^{D}&O\\X_{1}&I_{n-l}\end{bmatrix}U^*,\]
where $X_{1}=C_{L}(\sum_{i= 0}^{k-1}{A_{L}}^{i}{A_{L}}^{\pi}-{A_{L}}^{D})\text{ and }k=\operatorname{ind}(A_{L})$.
By using (\ref*{E1}), we can get (\ref*{E4})}.
\end{lemma}  

\begin{lemma}\label{mab}\emph{\cite{ma2019characterizations}}
  For $A,B,D,E\in\mathbb{C}^{n\times{n}}$ and $M=\begin{bmatrix} A& AD \\EA& B\end{bmatrix}\in\mathbb{C}^{2n\times{2n}}$,
then
  \[\textup{rank}{(M)}=\textup{rank}{(A)}+\textup{rank}{(B-EAD)}.\]
\end{lemma}

\begin{lemma} [Gramer's~rule]\label{Gramer}
  Let $A\in\mathbb{C}^{n \times n}$ be nonsingular and let $b\in\mathbb{C}^{n}$. Then the unique solution $x=\left[x_{1}, x_{2}, \ldots, x_{n}\right]^{T}$ of the system of linear equations $Ax=b$ is given by
  \[x_{i}=\frac{{\det} A(i \rightarrow b)}{{\det} A}, \quad i=1,2, \ldots, n \text {. }\]
\end{lemma}

\section{The properties of BDD-inverse}
In this section, we study the equivalence condition of $\textup{Ind}{(AP_{L}+P_{L^{\perp}})}=k$
and some basic properties of the BDD-inverse.
\begin{theorem}\label{T1}
Let $A\in\mathbb{C}^{n\times{n}}$, $L\leq\mathbb{C}^n$ and $k~(k\neq{1})$ is some positive integer. The following statements are equivalent:
\\$(a)$ $\textup{Ind}{(AP_{L}+P_{L^{\perp}})}=k$;
\\$(b)$ $\textup{Ind}{(P_{L}A+P_{L^{\perp}})}=k$;
\\$(c)$ $\textup{Ind}{(P_{L}AP_{L}+P_{L^{\perp}})}=k$;
\\$(d)$ $\textup{Ind}{(P_{L}AP_{L})}=k$;
\\$(e)$ $\textup{Ind}{(P_{L}A^{*}P_{L})}=k$.
\end{theorem} 

\begin{proof}
${(a)}\Leftrightarrow{(b)}$
It follows from (\ref{E5}) that
\[\textup{rank}{(AP_{L}+P_{L^\perp})} =\textup{rank}\begin{bmatrix} 
  {A_L} & O \\
  C_{L}& I_{n-l}
  \end{bmatrix}=\textup{rank}{(A_{L})}+\textup{rank}{(I_{n-l})},\]
  \[\textup{rank}{({(AP_{L}+P_{L^\perp})}^{2})} =\textup{rank}\begin{bmatrix} 
    {A_L}^{2} & O \\
    C_{L}A_{L}+C_{L}& I_{n-l}
    \end{bmatrix}=\textup{rank}{({A_{L}}^{2})}+\textup{rank}{(I_{n-l})}.\]
    Through mathematical induction, we get $\textup{rank}{({(AP_{L}+P_{L^\perp})}^{k})}=
    \textup{rank}{({A_{L}}^{k})}+\textup{rank}{(I_{n-l})}$ and $\textup{rank}{({(AP_{L}+P_{L^\perp})}^{k+1})}=
    \textup{rank}{({A_{L}}^{k+1})}+\textup{rank}{(I_{n-l})}$. Therefore, 
    $\textup{Ind}{(AP_{L}+P_{L^{\perp}})}=k\Leftrightarrow\textup{rank}{({A_{L}}^{k})}=\textup{rank}{({A_{L}}^{k+1})}$.
Similarly, we can conclude that $\textup{Ind}{(P_{L}A+P_{L^{\perp}})}=k\Leftrightarrow\textup{rank}{({A_{L}}^{k})}=\textup{rank}{({A_{L}}^{k+1})}\Leftrightarrow\textup{Ind}{(AP_{L}+P_{L^{\perp}})}=k$.\\
The rest of the proof follows similarly.
  \end{proof}

\noindent Now, we establish some interesting properties of BDD-inverse of $A\in\mathbb{C}^{n\times{n}}$.
\begin{theorem}\label{T2}
      Let $A\in\mathbb{C}^{n\times{n}}$, $L\leq\mathbb{C}^{n}$ and let $k=\textup{Ind}{(AP_{L}+P_{L^{\perp}})}$, $S=\mathcal{R}{({(P_{L}AP_{L})}^{k})}$ and $T=\mathcal{N}{({(P_{L}AP_{L})}^{k})}$, then the following statements hold:
    \vspace{0.2em}
    \\$(a)$ $A_{(L)}^{(D)}={P_L}A_{(L)}^{(D)}=A_{(L)}^{(D)}P_L={P_L}A_{(L)}^{(D)}P_{L}$;
    \\$(b)$ $\mathcal{R} {(A_{(L)}^{(D)})}=S$ and $\mathcal{N}{(A_{(L)}^{(D)})}=T$;
    \\$(c)$ $A_{(L)}^{(D)}=A_{(L)}^{(D)}AA_{(L)}^{(D)}=A_{(L)}^{(D)}A{P_L}A_{(L)}^{(D)}=A_{(L)}^{(D)}{P_L}AA_{(L)}^{(D)}=A_{(L)}^{(D)}{P_L}A{P_L}A_{(L)}^{(D)}$;
    \\$(d)$ $AA_{(L)}^{(D)}=P_{\mathcal{R}{({(AP_{L})}^{k+1})},T}~and~A_{(L)}^{(D)}A=P_{S,\mathcal{N}{({(P_{L}A)}^{k+1})}}$;
    \\$(e)$ $A_{(L)}^{(D)}AP_{L}=P_{L}AA_{(L)}^{(D)}=P_{S,T}$;
    \\$(f)$ $A_{(L)}^{(D)}=A_{S,T}^{(2)}={(AP_L)}_{S,T}^{(2)}={(P_{L}A)}_{S,T}^{(2)}={({P_L}A{P_L})}_{S,T}^{(2)}$;
    \\$(g)$ ${(A^{*})}_{(L)}^{(D)}={(A_{(L)}^{(D)})}^{*}$;
    \\$(h)$
    $A^{(D)}_{(L)}={({P_{L}AP_{L}})}^{D}=P_L{({P_L}A{P_L})}^{D}={({P_L}A{P_L})}^{D}P_L$\\
    $~~~~~~~~~~~~=P_L{({P_L}A{P_L} + P_{L^\perp})}^{D}={({P_L}A{P_L} + P_{L^\perp})}^{D}P_L$
    \\$~~~~~~~~~~~~={({P_L}A{P_L} + P_{L^\perp})}^{D}-P_{L^\perp}={(P_{L}A+P_{L^{\perp}})}^{D}P_{L}$
   \\$~~~~~~~~~~~~=P_{L}{(AP_{L})}^{D}={(P_{L}A)}^{D}P_{L}=P_{L}({(AP_{L})}^{2})^{D}AP_{L}$
    \\$~~~~~~~~~~~~=P_{L}A{({(P_{L}A)}^{2})}^{D}P_{L}$.
  \end{theorem}

 \begin{proof}
$(a)$.  Premultiplying ${(\ref{E1})}$ with $P_{L}$, we have $A_{(L)}^{(D)}={P_L}A_{(L)}^{(D)}$.
Using $(\ref{E2})$ and $(\ref{E4})$, we have
\begin{align*}
A^{(D)}_{(L)}P_{L}
&=U\begin{bmatrix}{A_{L}}^{D}&O\\O&O \end{bmatrix}
\begin{bmatrix}
  I_l & O \\
  O & O
\end{bmatrix}U^*\\
&=U\begin{bmatrix}{A_{L}}^{D}&O\\O&O \end{bmatrix}U^*=A^{(D)}_{(L)}.
\end{align*}
It follows from $A_{(L)}^{(D)}={P_L}A_{(L)}^{(D)}=A_{(L)}^{(D)}P_L$ that $A_{(L)}^{(D)}=P_{L}A_{(L)}^{(D)}P_L$.

\vspace*{0.5em}
\noindent $(b)$. By $(\ref{E1})$, we get
\begin{align*} 
  \mathcal{R}(A_{(L)}^{(D)})
  &=P_{L}\mathcal{R}{({(AP_L + P_{L^\perp})}^{D})}\\
  &=P_{L}\mathcal{R}{({(AP_L + P_{L^\perp})}^{k})}\\
  &=\mathcal{R}{({(P_{L}AP_{L})}^{k})}=S.
  \end{align*}
It follows that $\textup{rank}{(A_{(L)}^{(D)})}=\textup{rank}{({(P_{L}AP_{L})}^{k})}$.
From ${(\ref{E1})}$ and ${(\ref{dra})}$, we have 
\begin{align*}
\mathcal{N}{(A_{(L)}^{(D)})}&=\mathcal{N}{(P_{L}{(AP_L + P_{L^\perp})}^{D})}\\
&\subset\mathcal{N}{({(P_{L}AP_{L})}^{k+1}{(AP_L + P_{L^\perp})}^{D})}\\
&=\mathcal{N}{(P_{L}{{(AP_L + P_{L^\perp})}^{k+1}}{(AP_L + P_{L^\perp})}^{D})}\\
&=\mathcal{N}{(P_{L}{{(AP_L + P_{L^\perp})}^{k}})}\\
&=\mathcal{N}{({{(P_{L}AP_{L})}^{k}})}.
\end{align*}
Since $\mathcal{N}{(A_{(L)}^{(D)})}\subset\mathcal{N}{({{(P_{L}AP_{L})}^{k}})}=T$ and
$\textup{rank}{(A_{(L)}^{(D)})}=\textup{rank}{({(P_{L}AP_{L})}^{k})}$, it is easily obtained that $\mathcal{N}(A_{(L)}^{(D)})=T$.

\vspace*{0.5em}
\noindent $(c)$. From $(\ref{E3})$ and $(\ref*{E4})$, we infer that
\begin{align*}
A_{(L)}^{(D)}AA_{(L)}^{(D)}
&=U\begin{bmatrix}{A_{L}}^{D}&O\\O&O \end{bmatrix}\begin{bmatrix} 
    A_L & B_L \\
    C_L& E_L
    \end{bmatrix}\begin{bmatrix}{A_{L}}^{D}&O\\O&O \end{bmatrix}U^*\\
&=U\begin{bmatrix}{A_{L}}^{D}&O\\O&O \end{bmatrix}U^*=A_{(L)}^{(D)}.
\end{align*}
Using ${(a)}$ and $A_{(L)}^{(D)}=A_{(L)}^{(D)}AA_{(L)}^{(D)}$, we directly get
$A_{(L)}^{(D)}=A_{(L)}^{(D)}AA_{(L)}^{(D)}=A_{(L)}^{(D)}AP_{L}A_{(L)}^{(D)}
=A_{(L)}^{(D)}P_{L}AP_{L}A_{(L)}^{(D)}=A_{(L)}^{(D)}P_{L}AA_{(L)}^{(D)}$.

\vspace*{0.5em}
\noindent  ${(d)}$. From the first equation in ${(c)}$, we obtain $AA_{(L)}^{(D)}AA_{(L)}^{(D)}=AA_{(L)}^{(D)}$.
Also, ${(b)}$ implies that $\mathcal{R}{(AA_{(L)}^{(D)})}=A\mathcal{R}{(A_{(L)}^{(D)})}=AS=\mathcal{R}{({(AP_{L})}^{k+1})}$ and
$\mathcal{N}{(AA_{(L)}^{(D)})}=\mathcal{N}{(A_{(L)}^{(D)})}=T$.
Thus $AA_{(L)}^{(D)}=P_{\mathcal{R}{({{(AP_{L})}^{k+1}})},T}$.
Similarly, we have $A_{(L)}^{(D)}A=P_{S,{({A^*}T^\perp)}^\perp}$.
Now, from the equality ${{(A^{*}T^\perp)}^\perp}={{(A^{*}\mathcal{R}{({(P_{L}A^{*}P_{L})}^{k})})}^{\perp}}=
{\mathcal{R}{({(A^{*}P_{L})}^{k+1})}^{\perp}}=\mathcal{N}{({(P_{L}A)}^{k+1})}$,
it follows that $A_{(L)}^{(D)}A=P_{S,\mathcal{N}{({(P_{L}A)}^{k+1})}}$.

\vspace*{0.5em}
\noindent ${(e)}$. 
By ${(c)}$, we have
\[P_{L}AA_{(L)}^{(D)}P_{L}AA_{(L)}^{(D)}=P_{L}AA_{(L)}^{(D)}\]and
\[\mathcal{N}(A_{(L)}^{(D)})=\mathcal{N}(A_{(L)}^{(D)}P_{L}AA_{(L)}^{(D)})\geq\mathcal{N}(P_{L}AA_{(L)}^{(D)}).\]
Since $\mathcal{N}{(P_{L}AA_{(L)}^{(D)})}\geq\mathcal{N}{(A_{(L)}^{(D)})}$
and ${(b)}$, we get $\mathcal{N}{(P_{L}AA_{(L)}^{(D)})}=\mathcal{N}{(A_{(L)}^{(D)})}=T$
and $\textup{rank}{(P_{L}AA_{(L)}^{(D)})}=\textup{rank}{({(P_{L}AP_{L})}^{k})}$.
Using $(\ref{E1})$, we have
\begin{align*}
\mathcal{R}{(P_{L}AA_{(L)}^{(D)})}&=P_{L}AP_{L}\mathcal{R}{({(AP_L + P_{L^\perp})}^{D})}\\
&=P_{L}AP_{L}\mathcal{R}{({(AP_L + P_{L^\perp})}^{k})}\\
&=\mathcal{R}{({(P_{L}AP_{L})}^{k+1})}
\end{align*}
Considering $\mathcal{R}{(P_{L}AA_{(L)}^{(D)})}\leq\mathcal{R}{({(P_{L}AP_{L})}^{k})}$
and $\textup{rank}{(P_{L}AA_{(L)}^{(D)})}=\textup{rank}{({(P_{L}AP_{L})}^{k})}$,
it implies that $\mathcal{R}{(P_{L}AA_{(L)}^{(D)})}=S$.
Hence $P_{L}AA_{(L)}^{(D)}=P_{S,T}$. Analogously, we can prove $A_{(L)}^{(D)}AP_{L}=P_{S,T}$.

\vspace{0.5em}
\noindent ${(f)}$. It is apparent from ${(a)}$, ${(b)}$ and ${(c)}$.

\vspace{0.5em}
\noindent ${(g)}$. By ${(f)}$, we have $A_{(L)}^{(D)}=A_{S,T}^{(2)}$.
Thus, ${(A^{*})}_{(L)}^{(D)}={(A^{*})}_{\mathcal{R}{({(P_{L}{A^*}P_{L})}^{k})},\mathcal{N}{({(P_{L}{A}^{*}P_{L})}^{k})}}^{(2)}={(A^{*})}_{T^{\perp},S^{\perp}}^{(2)}={(A_{S,T}^{(2)})}^{*}={(A_{(L)}^{(D)})}^{*}$.

\vspace{0.5em}
\noindent ${(h)}$. Notice that for $\textup{Ind}{(B)}=k$, $B^{D}=B_{\mathcal{R}{(B^{k})},\mathcal{N}{(B^{k})}}^{(2)}$ (see {\cite{wang2018generalized}}).
Using the fourth equation in $(f)$, we have $A_{(L)}^{(D)}={(P_{L}AP_{L})}_{S,T}^{(2)}={(P_{L}AP_{L})}^{D}$.
From ${(a)}$, we know that $A_{(L)}^{(D)}={(P_{L}AP_{L})}^{D}=P_{L}{(P_{L}AP_{L})}^{D}={(P_{L}AP_{L})}^{D}P_{L}$.
Furthermore, using  ${(\ref{E1})}$ we have that 
\begin{align*}
A_{(L)}^{(D)}&={(P_{L}AP_{L})}^{D}={(P_{L}(P_{L}AP_{L})P_{L})}^{D}\\
&={({P_L}A{P_L})}^{(D)}_{(L)}=P_{L}{({P_L}A{P_L} + P_{L^\perp})}^{D}.
\end{align*}
By Lemma $\ref{pq0}$, we can verify that ${({P_L}A{P_L} + P_{L^\perp})}^{D}={({P_L}A{P_L})}^{D}+ P_{L^\perp}$.
Therefore, $A_{(L)}^{(D)}={({P_L}A{P_L})}^{D}={({P_L}A{P_L} + P_{L^\perp})}^{D}-P_{L^\perp}$.
By ${(a)}$, we can deduce that 
$A_{(L)}^{(D)}=A_{(L)}^{(D)}P_{L}={({({P_L}A{P_L} + P_{L^\perp})}^{D}-P_{L^\perp})}P_{L}={({P_L}A{P_L} + P_{L^\perp})}^{D}P_{L}$.
 Based on $(\ref{E2})$, $(\ref{E3})$ and Lemma $\ref{tri}$ we inferred that
\begin{align*}
  {(P_{L}A+P_{L^{\perp}})}^{D}P_{L}
&=U{\begin{bmatrix} 
  A_L & B_L \\
  O& I_{n-l}
  \end{bmatrix}}^{D}\begin{bmatrix}
    I_l & O \\
    O & O
    \end{bmatrix}U^*\\
&=U\begin{bmatrix} 
  {A_L}^{D} & X_2 \\
  O& I_{n-l}
  \end{bmatrix}\begin{bmatrix}
    I_l & O \\
    O & O
    \end{bmatrix}U^*\\
&= U\begin{bmatrix} 
  {A_L}^{D} & O \\
  O& O
  \end{bmatrix}U^*=A_{(L)}^{(D)},
\end{align*}
where $X_{2}={A_{L}}^{\pi}\sum_{i=0}^{k-1}{A_{L}}^{i}B_{L}-{A_{L}}^{D}B_{L}$.

\par Let $P_{L}$, $A$ and $A_{(L)}^{(D)}$ are expressed as in (\ref{E2}), (\ref{E3}) and (\ref{E4}), respectively.
From Lemma {\ref{aco}}, we derive
\begin{align*}
P_{L}{(AP_{L})}^{D}&=U \begin{bmatrix}
  I_l & O \\
  O & O\end{bmatrix}
  {\begin{bmatrix}
    A_{L} & O\\
    C_{L} & O
    \end{bmatrix}}^{D}U^*\\
&=U \begin{bmatrix}
  I_l & O \\
  O & O\end{bmatrix}
  \begin{bmatrix}
    {A_{L}}^{D} & O \\
    C_{L}{({A_{L}}^{D})}^{2} & O
    \end{bmatrix}U^*\\
&=U\begin{bmatrix}
  {A_{L}}^{D} & O \\
  O & O
  \end{bmatrix}U^*=A_{(L)}^{(D)}
  \end{align*}
and
\begin{align*}
  {(P_{L}A)}^{D}P_{L}&=U{\begin{bmatrix}
    A_{L} & B_{L}  \\
    O & O
    \end{bmatrix}}^{D}\begin{bmatrix}
      I_l & O \\
      O & O\end{bmatrix}U^*\\
  &=U\begin{bmatrix}
    {A_{L}}^{D} & {({A_{L}}^{D})}^{2}B_{L}  \\
    O & O
    \end{bmatrix}\begin{bmatrix}
      I_l & O \\
      O & O\end{bmatrix}U^*\\
  &=U\begin{bmatrix}
    {A_{L}}^{D} & O \\
    O & O
    \end{bmatrix}U^*=A_{(L)}^{(D)}.
\end{align*}    
We note that from Lemma {\ref{ccc}} it is easy to obtain
$A_{(L)}^{(D)}=P_{L}{({(AP_{L})}^{2})}^{D}AP_{L}=P_{L}A{({(P_{L}A)}^{2})}^{D}P_{L}$.
\end{proof}


\section{Some characterizations of the Bott-Duffin Drazin inverse}
\noindent In this section, we mainly discuss the characteristics of the BDD-inverse from the aspects of range space, null space, matrix equations and projectors.
\begin{theorem}\label{T3}
Let $A \in \mathbb{C}^{n\times n}$, $L \leq  \mathbb{C}^n$, $k=\textup{Ind}{(AP_{L}+P_{L^{\perp}})}$, $S=\mathcal{R}{({(P_{L}AP_{L})}^{k})}$, $T=\mathcal{N}{({(P_{L}AP_{L})}^{k})}$ and let $X \in \mathbb{C}^{n\times n}$. The following statements are equivalent:
\\$(a)~~X=A_{(L)}^{(D)}$;
\\$(b)~~\mathcal{R}(X)=S$, $AX=P_{\mathcal{R}{({(AP_{L})}^{k+1})},T}$;
\\$(c)~~\mathcal{R}(X)=S$, $P_{L}AX=P_{S,T}$;
\\$(d)~~\mathcal{R}(X)=S$, $XA=P_{S,\mathcal{N}{({(P_{L}A)}^{k+1})}}$, $XP_{T^{\perp}}=X$.
\end{theorem}

\begin{proof}
  ${(a)}\Rightarrow{(b)}$. It is evident from Theorem $\ref{T2}$ ${(b)}$ and ${(d)}$. 

  \noindent ${(b)}\Rightarrow{(c)}$. By ${(d)}$ and ${(e)}$ from Theorem $\ref{T2}$, we have $P_{L}AX = P_{L}P_{\mathcal{R}{({(AP_{L})}^{k+1})},T}= P_{S,T}$.
  
  \noindent  ${(c)}\Rightarrow{(d)}$. It follows from $\mathcal{R}{(X)}=S$ and $P_{L}AX = P_{S,T}$ that $\textup{rank}{(X)}=\dim{(S)}$ and $\mathcal{N}{(X)}=T$.
  Since $\mathcal{N}{(X)}=T$ and $L^{\perp}\leq{T}$, we conclude  that $XP_{T^{\perp}}=XP_{L}=X$. Therefore, $XAX = XP_{L}AX = XP_{S,T} = X$.
  From $XAX = X$, it is clear that $XAXA = XA$ and $\mathcal{R}{(XA)}=\mathcal{R}{(X)}=S$. Now, it follows from $\mathcal{N}{(X)}=T$ that $\mathcal{N}{(XA)}={(\mathcal{R}{({A^{*}X^{*}})})}^{\perp}={(A^{*}\mathcal{N}{(X)}^{\perp})}^{\perp}={(A^{*} T^{\perp})}^{\perp}=\mathcal{N}{({(P_{L}A)}^{k+1})}$.
  As a consequence, $XA=P_{S,\mathcal{N}{({(P_{L}A)}^{k+1})}}$.
  
  \noindent ${(d)}\Rightarrow(a)$. Since $\mathcal{R}{(X)}=S$ and $XA=P_{S,\mathcal{N}{({(P_{L}A)}^{k+1})}}$, we have  $XAX=X$. From $\textup{rank}{(X)}=\dim(S)$ and $XP_{T^{\perp}}=X$,
  we get  $\mathcal{N}{(X)}=T$. Finally By ${(f)}$ in Theorem $\ref{T2}$, it follows that $X=A_{(L)}^{(D)}$.
\end{proof}

\begin{theorem}\label{T4}
Let $A \in \mathbb{C}^{n\times n}$, $L \leq  \mathbb{C}^n$, $k=\textup{Ind}{(AP_{L}+P_{L^{\perp}})}$, $S=\mathcal{R}{({(P_{L}AP_{L})}^{k})}$, $T=\mathcal{N}{({(P_{L}AP_{L})}^{k})}$ and let $X \in \mathbb{C}^{n\times n}$. The following statements are equivalent:
\\$(a)~~X=A_{(L)}^{(D)}$;
\\$(b)~~\mathcal{N}{(X)}=T$, $XA=P_{S,\mathcal{N}{({(P_{L}A)}^{k+1})}}$;
\\$(c)~~\mathcal{N}{(X)}=T$, $XA{P_L}=P_{S,T}$;
\\$(d)~~\mathcal{N}{(X)}=T$, $AX=P_{\mathcal{R}{({(AP_{L})}^{k+1})},T}$, $P_{S}X=X$.
\end{theorem}

\begin{proof}
  ${(a)}\Rightarrow{(b)}$. It is obvious by ${(b)}$ and ${(d)}$ of Theorem $\ref{T2}$ .

  \noindent ${(b)}\Rightarrow{(c)}$. By ${(d)}$ and ${(e)}$ in Theorem $\ref{T2}$, it follows that $XAP_{L} =P_{S,\mathcal{N}{({(P_{L}A)}^{k+1})}} P_{L}= P_{S,T}$.
  
  \noindent ${(c)}\Rightarrow{(d)}$. Notice that $\mathcal{N}{(X)}=T$ and $XA{P_L}=P_{S,T}$,
  we have $\textup{rank}{(X)}=n-\dim{(T)}=\dim{(S)}$ and $\mathcal{R}{(X)}=S$, which means $P_{S}X=P_{L}X=X$. Thus, $XAX = XAP_{L}X = P_{S,T}X = X$. From $XAX = X$, it is clear that $AXAX = AX$ and $\mathcal{N}{(AX)}=\mathcal{N}{(X)}=T$. Now, by  $\mathcal{R}{(X)}=T$ we  have that $\mathcal{R}{(AX)}=A\mathcal{R}{(X)}=AS=\mathcal{R}{({(AP_{L})}^{k+1})}$.
  Hence, $AX=P_{\mathcal{R}{({(AP_{L})}^{k+1})},T}$.
  
  \noindent ${(d)}\Rightarrow{(a)}$. Since  $\mathcal{N}(X)=T$ and $AX=P_{\mathcal{R}{({(AP_{L})}^{k+1})},T}$, it follows that $XAX=X$. From $\textup{rank}{(X)}=\dim(S)$ and $P_{S}X=X$,
  we have $\mathcal{R}{(X)}=S$. Now by Theorem $\ref{T2}$ ${(f)}$, we get $X=A_{(L)}^{(D)}$.
  \end{proof}

\noindent By Theorem {\ref{T2}} $(c)$, we observe that $A_{(L)}^{(D)}$ is the $\{2\}$-inverse of $A$.
As a consequence, the following results can be obtained.
\begin{theorem}\label{T5}
  Let $A \in \mathbb{C}^{n\times n}$, $L \leq  \mathbb{C}^n$, $k=\textup{Ind}{(AP_{L}+P_{L^{\perp}})}$, $S=\mathcal{R}{({(P_{L}AP_{L})}^{k})}$, $T=\mathcal{N}{({(P_{L}AP_{L})}^{k})}$ and let $X \in \mathbb{C}^{n\times n}$. The following statements are equivalent:
  \\$(a)$ $X=A_{(L)}^{(D)}$;
  \\$(b)$ $XAX=X$, $XP_{T^{\perp}}=X$ and $XA=P_{S,\mathcal{N}{({(P_{L}A)}^{k+1})}}$;
  \\$(c)$ $XAX=X$, $XAP_{L}=P_{S,T}$ and $AX=P_{\mathcal{R}{({(AP_{L})}^{k+1})},T}$; 
  \\$(d)$ $XAX=X$, $P_{S}X=X$ and $AX=P_{\mathcal{R}{({(AP_{L})}^{k+1})},T}$;
  \\$(e)$ $XAX=X$, $P_{L}AX=P_{S,T}$ and $XA=P_{S,\mathcal{N}{({(P_{L}A)}^{k+1})}}$;
  \\$(f)$ $XAX=X$, $P_{S}XP_{T^{\perp}}=X$ and $\textup{rank}(X)=\dim(S)$.
\end{theorem}

\begin{proof}
  ${(a)}\Rightarrow{(b)}$.  It is a direct consequence of Theorem $\ref{T2}$ ${(b)}$, ${(c)}$ and ${(d)}$ .

  \noindent  ${(b)}\Rightarrow{(c)}$. From ${(d)}$ and ${(e)}$ from Theorem $\ref{T2}$, we have $XAP_{L}=P_{S,\mathcal{N}{({(P_{L}A)}^{k+1})}}P_{L}=P_{S, T}$. It follows from $XAX=X$ that $AXAX=AX$, $\mathcal{R}{(XA)}=\mathcal{R}{(X)}$ and $\mathcal{N}{(AX)}=\mathcal{N}{(X)}$.
  Since $XA=P_{S,\mathcal{N}{({(P_{L}A)}^{k+1})}}$, we obtain $\mathcal{R}{(X)}=\mathcal{R}{(XA)}=S$, which implies $\mathcal{R}{(AX)} = AS=\mathcal{R}{({(AP_{L})}^{k+1})}$ and $\text{rank}{(X)}=\text{rank}{((P_{L}AP_{L})^{k})}$.
  By $XP_{T^{\perp}}=X$, we get $\mathcal{N}(X)=T$. Thus, $AX=P_{\mathcal{R}{({(AP_{L})}^{k+1})},T}$.
  
  \noindent   ${(c)}\Rightarrow{(d)}$. It follows from $XAX =X$ and $AX=P_{\mathcal{R}{({(AP_{L})}^{k+1})},T}$ that $\mathcal{N}{(X)}=\mathcal{N}{(AX)}=T$,
  which implies $\textup{rank}{(X)} = n-\dim{(T)}=\dim{(S)}$. By $XAP_{L} = P_{S,T}$, we conclude that
  $S\leq\mathcal{R}{(X)}$. Hence, $\mathcal{R}{(X)}=S$, which implies $P_{S}X = X$.
  
  \noindent  ${(d)}\Rightarrow{(e)}$. The proof is similar to the proof of ${(b)}\Rightarrow{(c)}$.
  
  \noindent  ${(e)}\Rightarrow{(f)}$. From $XAX=X$ and $XA=P_{S,\mathcal{N}{({(P_{L}A)}^{k+1})}}$, we have $\mathcal{R}{(X)}=S$, which gives
  $P_{S} X=X$ and $\textup{rank}{(X)}=\dim(S)$. By $P_{L}AX=P_{S,T}$ and $\mathcal{R}{(X)}=S$ it follows that $\mathcal{N}{(X)}=T$, which implies $XP_{T^{\perp}}=X$. So, $P_{S}XP_{T^{\perp}}=XP_{T^{\perp}}=X$.
  
  \noindent  ${(f)}\Rightarrow{(a)}$. By $P_{S}XP_{T^{\perp}}=X$ and $\textup{rank}{(X)}=\dim{(S)}$, we can derive that $\mathcal{R}{(X)}={S}$ and  $\mathcal{N}{(X)}=T$.
  According to Theorem $\ref{T2}$ ${(f)}$, we obtain $X=A_{(L)}^{(D)}$.\\
  \end{proof}

\noindent From Theorem {\ref{T2}} $(d)$, we observe that
\[X=A_{(L)}^{(D)}\Rightarrow AX=P_{\mathcal{R}{({(AP_{L})}^{k+1})},T}~\text{and}~XA=P_{S,\mathcal{N}{({(P_{L}A)}^{k+1})}}.\]
However, the equations $AX=P_{\mathcal{R}{({(AP_{L})}^{k+1})},T}~\text{and}~XA=P_{S,\mathcal{N}{({(P_{L}A)}^{k+1})}}$
are not enough for $X=A_{(L)}^{(D)}$, and this can be confirmed by means
of the following example.
\begin{example}\label{ex1}
  We consider  \[A=\begin{bmatrix}
    3 & 3 & 3 & 2 \\
    1 & 2 & 2 & 3 \\
    2 & 1 & 1 & 3 \\
    0 & 0 & 0 & 0 \\
      \end{bmatrix},~
    L=\mathcal{R}\left(\begin{array}{lll}\begin{bmatrix}
      1 & 0 & 0 \\
      0 & 1 & 0 \\
      0 & 0 & 1 \\
      0 & 0 & 0\end{bmatrix}
      \end{array}\right),~
    X=\begin{bmatrix}
      \frac{1}{12} & \frac{1}{12} & \frac{1}{12} & 0 \\
      \frac{1}{24} & \frac{1}{24} & \frac{1}{24} & 1 \\
      \frac{1}{24} & \frac{1}{24} & \frac{1}{24} & -1 \\
      0 & 0 & 0 & 0 \\
      \end{bmatrix}.\]
  It is easy to check that $\textup{rank}{(P_{L}AP_{L})}=\textup{Ind}{(P_{L}AP_{L})}=2$.
  By computations, we have
\[A_{(L)}^{(D)}=\begin{bmatrix}\frac{1}{12} & \frac{1}{12} & \frac{1}{12} & 0 \\
  \frac{1}{24} & \frac{1}{24} & \frac{1}{24} & 0 \\
  \frac{1}{24} & \frac{1}{24} & \frac{1}{24} & 0 \\
  0 & 0 & 0 & 0 \\
\end{bmatrix},~
AX=\begin{bmatrix}
  \frac{1}{2} & \frac{1}{2} & \frac{1}{2} & 0 \\
  \frac{1}{4} & \frac{1}{4} & \frac{1}{4} & 0 \\
  \frac{1}{4} & \frac{1}{4} & \frac{1}{4} & 0 \\
  0 & 0 & 0 & 0 \\
  \end{bmatrix},~
 XA=\begin{bmatrix}
  \frac{1}{2} & \frac{1}{2} & \frac{1}{2} & \frac{2}{3} \\
  \frac{1}{4} & \frac{1}{4} & \frac{1}{4} & \frac{1}{3} \\
  \frac{1}{4} & \frac{1}{4} & \frac{1}{4} & \frac{1}{3} \\
  0 & 0 & 0 & 0 \\
  \end{bmatrix}.\]
  We can immediately verified that $AX=P_{\mathcal{R}{({(AP_{L})}^{k+1})},T}$ and $XA=P_{S,\mathcal{N}{({(P_{L}A)}^{k+1})}}$,
  but $X\neq A_{(L)}^{(D)}$.
\end{example}

\noindent The next theorem offers a characterization of BDD-inverse by using the
$AX=P_{\mathcal{R}{({(AP_{L})}^{k+1})},T}$, $XA=P_{S,\mathcal{N}{({(P_{L}A)}^{k+1})}}$.

\begin{theorem}\label{T7}
  Let $A \in \mathbb{C}^{n\times n}$, $L \leq  \mathbb{C}^n$, $k=\textup{Ind}{(AP_{L}+P_{L^{\perp}})}$, $S=\mathcal{R}{({(P_{L}AP_{L})}^{k})}$, $T=\mathcal{N}{({(P_{L}AP_{L})}^{k})}$ and let $X \in \mathbb{C}^{n\times n}$. The following statements are equivalent:
\\$(a)$ $X=A_{(L)}^{(D)}$;
\\$(b)$ $AX=P_{\mathcal{R}{({(AP_{L})}^{k+1})},T}$, $XA=P_{S,\mathcal{N}{({(P_{L}A)}^{k+1})}}$ and $XAX=X$;
\\$(c)$ $AX=P_{\mathcal{R}{({(AP_{L})}^{k+1})},T}$, $XA=P_{S,\mathcal{N}{({(P_{L}A)}^{k+1})}}$ and $\operatorname{rank}(\mathrm{X})=\dim({S})$;
\\$(d)$ $AX=P_{\mathcal{R}{({(AP_{L})}^{k+1})},T}$, $XA=P_{S,\mathcal{N}{({(P_{L}A)}^{k+1})}}$ and $XP_{T^{\perp}}=X$;
\\$(e)$ $AX=P_{\mathcal{R}{({(AP_{L})}^{k+1})},T}$, $XA=P_{S,\mathcal{N}{({(P_{L}A)}^{k+1})}}$ and $P_{S}X=X$;
\\$(f)$ $AX=P_{\mathcal{R}{({(AP_{L})}^{k+1})},T}$, $XA=P_{S,\mathcal{N}{({(P_{L}A)}^{k+1})}}$ and $P_{S}XP_{T^{\perp}}=X$.
\end{theorem}

\begin{proof}
  ${(a)}\Rightarrow{(b)}$. This follows  by ${(d)}$ and the first equation of ${(c)}$ in Theorem $\ref{T2}$.
  
  \noindent $(b)\Rightarrow(c)$. From $XA=P_{S,\mathcal{N}{({(P_{L}A)}^{k+1})}}$ and $XAX=X$, we obtain $\mathcal{R}{(X)}=\mathcal{R}{(XA)}=S$.
  Hence, $\textup{rank}{(X)}=\dim({S})$.
  
  \noindent ${(c)}\Rightarrow{(d)}$. By $AX=P_{\mathcal{R}{({(AP_{L})}^{k+1})},T}$ and $\textup{rank}{(X)}=\dim({S})$, it follows that $\mathcal{N}{(X)}=T$, which means
  $XP_{T^{\perp}}=X$.
  
  \noindent ${(d)}\Rightarrow{(e)}$. Since $AX=P_{\mathcal{R}{({(AP_{L})}^{k+1})},T}$ and $XP_{T^{\perp}}=X$, we get $\mathcal{N}{(X)}=T$ and $\textup{rank}{(X)}=\dim{(S)}$. From $XA=P_{S,\mathcal{N}{({(P_{L}A)}^{k+1})}}$, we can infer that $\mathcal{R}{(X)}=S$. Hence $P_{S}X=X$.
  
  \noindent ${(e)}\Rightarrow{(f)}$. It follows from $XA=P_{S,\mathcal{N}{({(P_{L}A)}^{k+1})}}$ and $P_{S}X=X$ that $\mathcal{R}{(X)}=S$. Furthermore, from $AX=P_{\mathcal{R}{({(AP_{L})}^{k+1})},T}$, we can conclude that $\mathcal{N}{(X)}=T$, which gives $XP_{T^{\perp}}=X$. Therefore, we get $P_{S}XP_{T^{\perp}}=XP_{T^{\perp}}=X$.
 
  \noindent ${(f)}\Rightarrow{(a)}$. Using that $XA=P_{S,\mathcal{N}{({(P_{L}A)}^{k+1})}}$ and $P_{S}XP_{T^{\perp}}=X$, we have that $XAX=XAP_{S}XP_{T^{\perp}}=P_{S}XP_{T^{\perp}}=X$.
  Hence, we get $\mathcal{R}{(X)}=\mathcal{R}{(XA)}=S$ and $\mathcal{N}{(X)}=\mathcal{N}{(AX)}$.
  Since $AX=P_{\mathcal{R}{({(AP_{L})}^{k+1})},T}$, we have $\mathcal{N}{(X)}=T$. From $(f)$ of  Theorem $\ref{T2}$, we obtain $X=A_{(L)}^{(D)}$.
  \end{proof}

\noindent Now, we will give certain characterization of the BDG-inverse of $A\in\mathbb{C}^{n\times{n}}$
  \begin{theorem}\label{T6}
    Let $A \in\mathbb{C}^{n\times{n}}$, $L\leq \mathbb{C}^n$, $k=\textup{Ind}{(AP_{L}+P_{L^{\perp}})}$, $S=\mathcal{R}{({(P_{L}AP_{L})}^{k})}$, $T=\mathcal{N}{({(P_{L}AP_{L})}^{k})}$ and let $X \in \mathbb{C}^{n\times n}$. The following statements are equivalent:
    \\$(a)$ $X=A_{(L)}^{(D)}$;
    \\$(b)$ $AX=P_{\mathcal{R}{({(AP_{L})}^{k+1})},T}$ and $P_{S}X=X$;
    \\$(c)$ $XA=P_{S,\mathcal{N}{({(P_{L}A)}^{k+1})}}$ and $XP_{T^{\perp}}=X$;
    \\$(d)$ $P_{L}AX=P_{S,T}$ and $P_{S} X=X$;
    \\$(e)$ $P_{L}AX=P_{S,T}$ and $P_{S}XP_{T^{\perp}}=X$;
    \\$(f)$ $XAP_{L}=P_{S, T}$ and $XP_{T^{\perp}}=X$;
    \\$(g)$ $XAP_{L}=P_{S,T}$ and $P_{S}XP_{T^{\perp}}=X$;
    \end{theorem}
    
    \begin{proof}
      ${(a)}\Rightarrow{(b)}$. This follows  directly by $(b)$ and $(d)$ of Theorem $\ref{T2}$. 
    
    \noindent ${(b)}\Rightarrow{(a)}$. The conditions $AX=P_{\mathcal{R}{({(AP_{L})}^{k+1})},T}$ and $P_{S}X=X$ imply that $\mathcal{N}{(X)}\leq{T}$ and $\mathcal{R}{(X)}\leq{S}$, respectively. So, we have $\mathcal{R}{(X)}={S}$ and $\mathcal{N}{(X)}={T}$.
    Hence $XAX=XP_{L}AX=XP_{S,T}=X$. By ${(f)}$ of Theorem $\ref{T2}$, we get $X=A_{(L)}^{(D)}$.
    
    \noindent  The rest of the proof follows similarly.
    \end{proof}

\noindent The following theorem gives the characteristics of $A_{(L)}^{(D)}$ via matrix decomposition.
\begin{theorem}
  Let $A \in\mathbb{C}^{n\times{n}}$, $L\leq\mathbb{C}^n$, $k=\textup{Ind}{(AP_{L}+P_{L^{\perp}})}$ and let $X\in\mathbb{C}^{n\times{n}}$. The following statements are equivalent:
  \\$(a)$ $X=A_{(L)}^{(D)}$;
  \\$(b)$ $X^{2}AP_{L}=X$, ${(P_{L}AP_{L})}X=XAP_{L}$ and ${(P_{L}A)}^{k+1}X={(P_{L}A)}^{k}P_{L}$;
  \\$(c)$ $P_{L}AX^{2}=X$, $P_{L}AX=X{(P_{L}AP_{L})}$ and ${(P_{L}A)}^{k+1}X={(P_{L}A)}^{k}P_{L}$;
  \\$(d)$ $X^{2}AP_{L}=P_{L}AX^{2}=X$, $P_{L}AX=XAP_{L}$ and ${(P_{L}A)}^{k+1}X={(P_{L}A)}^{k}P_{L}$.
\end{theorem}

\begin{proof}
Let $P_{L}$, $A$ and $A_{(L)}^{(D)}$ are given by (\ref{E2}), (\ref{E3}) and (\ref{E4}), respectively
By simple calculation,  ${(b)}-{(d)}$ can be derived from ${(a)}$.
\par ${(b)}\Rightarrow{(a)}$.
  Suppose that $X$ be partitioned as
\begin{equation}\label{X}
  X= U \begin{bmatrix} 
    X_{11} & X_{12} \\
    X_{21}& X_{22}
    \end{bmatrix} U^*,
\end{equation}
where $X_{11}\in\mathbb{C}^{l\times{l}}$, $X_{12}\in\mathbb{C}^{l\times{(n-l)}}$, $X_{21}\in\mathbb{C}^{{(n-l)}\times{l}}$ and $X_{22}\in\mathbb{C}^{{(n-l)}\times{(n-l)}}$.
Let $P_{L}$ and $A$ are given by (\ref{E2}) and (\ref{E3}), respectively.
From $X^{2}AP_{L}=X$, we obtain ${X_{11}}^{2}A_{L}=X_{11}$, $X_{21}X_{11}A_{L}=X_{21}$ and $X_{12}=X_{22}=O$.
In this case, by ${(P_{L}AP_{L})}X=XAP_{L}$ we infer to $A_{L}X_{11}=X_{11}A_{L}$ and $X_{21}A_{L}=O$, it follows that $X_{21}=X_{21}X_{11}A_{L}=X_{21}A_{L}X_{11}=O$.
In terms of  ${(P_{L}A)}^{k+1}X={(P_{L}A)}^{k}P_{L}$, we get ${A_{L}}^{k+1}X_{11}={A_{L}}^{k}$.
Through (\ref{dra}) and (\ref{E4}), we can verify that $X=A_{(L)}^{(D)}$.\\
\noindent The rest of the proof follows similarly.
\end{proof}
\section{ Some representations of the Bott-Duffin drazin inverse}
As we all know, for a invertible matrix $A\in\mathbb{C}^{n\times{n}}$, the inverse matrix $A^{-1}$ of $A$ is the unique matrix $X\in\mathbb{C}^{n\times{n}}$, which satisfies 
\begin{equation*}\left.\left.\operatorname{rank}\left(\left[\begin{array}{cc}A&I_n\\I_n&X\end{array}\right.\right.\right]\right)=\operatorname{rank}\left(A\right).\end{equation*}

\noindent In this section, we generalize this fact to singular matrices $A$ to obtain similar results for the BDD-inverse $A_{(L)}^{(D)}$ of $A\in\mathbb{C}^{n\times{n}}$.

\begin{theorem}\label{repre}
Let $A\in\mathbb{C}^{n\times{n}}$, $L\leq\mathbb{C}^{n}$, $\textup{rank}{({(P_{L}AP_{L})}^{k})}=r$ and $\textup{Ind}{(P_{L}AP_{L})}=k$.
Also, suppose that $S=\mathcal{R}{((P_{L}AP_{L})^{k})}$, $T=\mathcal{N}{((P_{L}AP_{L})^{k})}$.
Then there are a unique matrix $W_{1}\in\mathbb{C}^{n\times{n}}$ satisfying
\begin{equation}\label{WY1}
  {W_{1}}^{2}=W_{1}, \quad W_{1}{(P_{L}A)}^{k+1}=O, \quad {(P_{L}A)}^{k+1}W_{1}=O, \quad \operatorname{rank}(W_{1})=n-r
\end{equation}
a unique matrix $W_{2}\in\mathbb{C}^{n\times{n}}$ satisfying
\begin{equation}\label{WPA}
  {W_{2}}^{2}=W_{2}, \quad W_{2}{(AP_{L})}^{k+1}=O, \quad {(AP_{L})}^{k+1}W_{2}=O, \quad \operatorname{rank}(W_{2})=n-r
\end{equation}
and a unique matrix $X\in\mathbb{C}^{n\times{n}}$ satisfying
\begin{equation}\label{WY2}
\operatorname{rank}\left(\left[\begin{array}{cc}
A & I_{n}-W_{2} \\
I_{n}-W_{1} & X
\end{array}\right]\right)=\operatorname{rank}(A) .
\end{equation}
The matrix X is the BDD-inverse $A_{(L)}^{(D)}$ of A. Moreover, we have
\begin{equation}\label{WY3}
W_{1}=P_{\mathcal{N}{({(P_{L}A)}^{k+1})},S},\quad W_{2}=P_{T,\mathcal{R}{({(AP_{L})}^{k+1})}}.
\end{equation}
\end{theorem}

\begin{proof}
  Suppose that $P_{L}$ and $A$ are expressed by (\ref{E2}) and (\ref{E3}), respectively.
  By $\textup{Ind}{(P_{L}AP_{L})}=k$, we know that
  \[\textup{rank}({(P_{L}A)}^{k+1})= \textup{rank}\left(\left[\begin{array}{cc}
    {A_{L}}^{k+1} & {A_{L}}^{k}B_{L} \\
    O& O\end{array}\right]\right)=\textup{rank}({A_{L}}^{k}),\]
    \[\textup{rank}({(P_{L}AP_{L})}^{k})= \textup{rank}\left(\left[\begin{array}{cc}
      {A_{L}}^{k} & O \\
      O& O\end{array}\right]\right)=\textup{rank}({A_{L}}^{k})\]
which means
$\mathcal{R}{({(P_{L}A)}^{k+1})}=\mathcal{R}{({(P_{L}AP_{L})}^{k})}$.
As a consequence, it can be straightforwardly derived that
\begin{align*}
  \text{the condition (\ref{WY1}) holds}\Longleftrightarrow\quad&{(I_{n}-W_{1})}^{2}=I_{n}-W_{1}, \quad 
  (I_{n}-W_{1}){(P_{L}A)}^{k+1} = {(P_{L}A)}^{k+1},\\
  &{(P_{L}A)}^{k+1}={(P_{L}A)}^{k+1}(I_{n}-W_{1}), \quad \textup{rank}{(I_{n}-W_{1})}=r\\
  \Longleftrightarrow\quad& {(I_{n}-W_{1})}^{2}=I_{n}-W_{1}, \quad \mathcal{R}{(I_{n}-W_{1})}={S},\\
  &\mathcal{N}{(I_{n}-W_{1})}=\mathcal{N}{({(P_{L}A)}^{k+1})}\\
  \Longleftrightarrow\quad& I_{n}-W_{1}=P_{S,\mathcal{N}{({(P_{L}A)}^{k+1})}}\\
  \Longleftrightarrow\quad& W_{1}=P_{\mathcal{N}{({(P_{L}A)}^{k+1})},S}.
\end{align*}
Since Theorem {\ref{T2}} ${(d)}$,  it follows from $\mathcal{N}{(P_{S})}\subset\mathcal{N}{(AP_{S})}\subset\mathcal{N}{(A_{(L)}^{(D)}AP_{S})}=\mathcal{N}{(P_{S})}$
that $\mathcal{N}{(P_{S})}=\mathcal{N}{(AP_{S})}$ and   $\textup{rank}{(AP_{S})}=\textup{rank}{(P_{S})}=r$.
Analogously, it can be verified that the condition (\ref{WPA}) has the unique solution $W_{2}=P_{T,\mathcal{R}{({(AP_{L})}^{k+1})}}$.
From this, it can be seen that (\ref{WY3}) is satisfied.
\par Finally, by Lemma {\ref{mab}}, it follows
from $(\ref{WY3})$ and Theorem $\ref{T2}$ ${(d)}$ that
\begin{align*}\label{WY4}
  \operatorname{rank}\left(\left[\begin{array}{cc}
    A & I_{n}-W_{2} \\
    I_{n}-W_{1} & X
    \end{array}\right]\right)&= \operatorname{rank}\left(\left[\begin{array}{cc}
    A &  AA_{(L)}^{(D)} \\
    A_{(L)}^{(D)}A & X
    \end{array}\right]\right)\\
    &=\textup{rank}{(A)}+\textup{rank}{(X-A_{(L)}^{(D)})}.
  \end{align*}
Thus,
\begin{align*}
\text{the condition (\ref{WY2}) holds}\Longleftrightarrow&\textup{rank}{(X-A_{(L)}^{(D)})}=O\\
\Longleftrightarrow&X=A_{(L)}^{(D)}.
\end{align*}
It follows that $X=A_{(L)}^{(D)}$ is the unique matrix satisfying $(\ref{WY2})$.
\end{proof}

\noindent  In the following result, we present another representation of $A_{(L)}^{(D)}$ by applying Theorem $\ref{repre}$.
We use $A[\alpha|\beta]$ to represent the submatrix composed of the row index set $\alpha$ and the column index set $\beta$ of the matrix $A$.
\begin{theorem}\label{wn1}
  Let $A\in\mathbb{C}^{n\times{n}}$, $L\leq\mathbb{C}^{n\times{n}}$ and $\operatorname{rank}{(A)}=r \geqslant{1}$, and let
  $N=\{1,2,\ldots,n\}$. If  $\alpha,\beta\subseteq{N}$ are such that ${A}[\alpha|\beta]\in\mathbb{C}^{r\times{r}}$  is reversible, then
\begin{equation}\label{co1}
  A_{(L)}^{(D)}=(I_{n}-W_{1})[N|\beta]{(A[\alpha|\beta])}^{-1}(I_{n}-W_{2})[\alpha|N] .
\end{equation}
\end{theorem}

\begin{proof}
  Suppose
 \[Y_{1}=\left[\begin{array}{cc}
  A & I_{n}-W_{2} \\
  I_{n}-W_{1} & A_{(L)}^{(D)}
  \end{array}\right]~\text{and}~Y_{2}=\left[\begin{array}{cc}
      A[\alpha|\beta] & (I_{n}-W_{2})[\alpha|N] \\
      (I_{n}-W_{1})[N|\beta] & A_{(L)}^{(D)}
      \end{array}\right].\]
  By Theorem $\ref{repre}$, we know that $\textup{rank}{(Y_{1})}=\textup{rank}{(A)}$. Therefore
\[r=\textup{rank}{(A[\alpha|\beta])}\leq\textup{rank}{(Y_{2})}\leq\textup{rank}{(Y_{1})}=\textup{rank}{(A)}=r,\]
which means
\begin{equation}\label{co2}
\textup{rank}{(Y_{2})}=\textup{rank}{(A[\alpha|\beta])}=r.
\end{equation}
Furthermore, by using basic row and column operations of block matrices, we obtain
\begin{equation}\label{co3}
  \textup{rank}{(Y_{2})}=\textup{rank}{(A[\alpha|\beta])}+\textup{rank}{\left(A_{(L)}^{(D)}-(I_{n}-W_{1})[N|\beta]{(A[\alpha|\beta])}^{-1}(I_{n}-W_{2})[\alpha|N]\right)}.
\end{equation}
Hence, (\ref{co1}) can be obtained from (\ref{co2}) and (\ref{co3}).
\end{proof}

\begin{example}\label{stm}
  Suppose that the matrix
\[A=\begin{bmatrix}
  1 & 1 & 1 & 1 \\
  0 & 1 & 2 & 3 \\
  1 & 1 & 1 & 1 \\
  1 & 1 & 1 & 1 \\
  \end{bmatrix},~
  L=\mathcal{R}\left(\begin{array}{lll}\begin{bmatrix}
    1 & 0 & 0 \\
    0 & 1 & 0 \\
    0 & 0 & 1 \\
    0 & 0 & 0\end{bmatrix}
    \end{array}\right).\]
It is easy to check that $\textup{Ind}{(P_{L}AP_{L})}=2$. Show through calculation that 
\[A_{(L)}^{(D)}=\begin{bmatrix}
  \frac{2}{27} & \frac{1}{9} & \frac{4}{27} & 0 \\
  \frac{2}{27} & \frac{1}{9} & \frac{4}{27} & 0 \\
  \frac{2}{27} & \frac{1}{9} & \frac{4}{27} & 0 \\
  0 & 0 & 0 & 0 \\
  \end{bmatrix},~
  AA_{(L)}^{(D)}=\begin{bmatrix}
    \frac{2}{9} & \frac{1}{3} & \frac{4}{9} & 0 \\
    \frac{2}{9} & \frac{1}{3} & \frac{4}{9} & 0 \\
    \frac{2}{9} & \frac{1}{3} & \frac{4}{9} & 0 \\
    \frac{2}{9} & \frac{1}{3} & \frac{4}{9} & 0 \\
    \end{bmatrix},~
    A_{(L)}^{(D)}A=\begin{bmatrix}
      \frac{2}{9} & \frac{1}{3} & \frac{4}{9} & \frac{5}{9} \\
      \frac{2}{9} & \frac{1}{3} & \frac{4}{9} & \frac{5}{9} \\
      \frac{2}{9} & \frac{1}{3} & \frac{4}{9} & \frac{5}{9} \\
      0 & 0 & 0 & 0 \\
      \end{bmatrix}.\]
Taking $\alpha=\{1,2\}$, $\beta=\{1,2\}$ and $N=\{1,2,3,4\}$, then
\[A[\alpha|\beta]=\begin{bmatrix}
  1 & 1    \\
  0 & 1    \\
  \end{bmatrix},~~~
{(I_{n}-W_{1})}[N|\beta]=\begin{bmatrix}
  \frac{2}{9} & \frac{1}{3}  \\
  \frac{2}{9} & \frac{1}{3}  \\
  \frac{2}{9} & \frac{1}{3}  \\
  0 & 0  \\
  \end{bmatrix},~
  {(I_{n}-W_{2})}[\alpha|N]=\begin{bmatrix}
    \frac{2}{9} & \frac{1}{3} & \frac{4}{9} & 0 \\
    \frac{2}{9} & \frac{1}{3} & \frac{4}{9} & 0 \\
    \end{bmatrix}.\]
It is obvious that $\textup{rank}{(A)}=\textup{rank}{(A[\alpha|\beta])}=2$, that is,
$A[\alpha|\beta]$ is reversible.
Hence, by direct calculation, we get
\begin{align*}
    (I_{n}-W_{1})[N|\beta]{(A[\alpha|\beta])}^{-1}(I_{n}-W_{2})[\alpha|N]
&=\begin{bmatrix}
  \frac{2}{27} & \frac{1}{9} & \frac{4}{27} & 0 \\
  \frac{2}{27} & \frac{1}{9} & \frac{4}{27} & 0 \\
  \frac{2}{27} & \frac{1}{9} & \frac{4}{27} & 0 \\
  0 & 0 & 0 & 0 \\
  \end{bmatrix}\\
  &=A_{(L)}^{(D)}.
\end{align*}    
\end{example}

\noindent Below, we use matrix decomposition to give some other representations of $A_{(L)}^{(D)}$.

\begin{theorem}\label{ww}
Let $A\in\mathbb{C}^{n\times{n}}$, $L\leq \mathbb{C}^{n}$. Then
\begin{align*}
A_{(L)}^{(D)}&=P_{L}{(I_{n}-W_{2})}{(AP_{L})}^{D}\\
&={(P_{L}A)}^{D}{(I_{n}-W_{1})}P_{L}\\
&={(I_{n}-W_{1})}A^{\dagger}{(I_{n}-W_{2})},
\end{align*}
where $W_{1},~W_{2}$ are defined as in (\ref{WY3}).
\end{theorem}

\begin{proof}
From ${(d)}$ of Theorem $\ref{T2}$, $(\ref{E2})$, $(\ref{E3})$ and $(\ref{E4})$, we can deduce that
\[I_{n}-W_{1}=U\begin{bmatrix}{A_{L}}^{D}A_{L}&{A_{L}}^{D}B_{L}\\O&O \end{bmatrix}U^*~
\textup{and}~I_{n}-W_{2}=U\begin{bmatrix}A_{L}{A_{L}}^{D}&O\\C_{L}{A_{L}}^{D}&O \end{bmatrix}U^*.\]
By Lemma $\ref{aco}$, we directly get
\begin{align*}
  P_{L}{(I_{n}-W_{2})}{(AP_{L})}^{D}&=U \begin{bmatrix}
  I_l & O \\
  O & O\end{bmatrix}\begin{bmatrix}A_{L}{A_{L}}^{D}&O\\C_{L}{A_{L}}^{D}&O \end{bmatrix}{\begin{bmatrix}{A_{L}}&O\\C_{L}&O\end{bmatrix}}^{D}U^*\\
&=U\begin{bmatrix}{A_{L}A_{L}}^{D}&O\\O&O \end{bmatrix}\begin{bmatrix}
  {A_{L}}^{D} & O \\
  C_{L}({A_{L}}^{D})^{2} & O
  \end{bmatrix}U^*\\
&=U\begin{bmatrix}{A_{L}}^{D}&O\\O&O \end{bmatrix}U^*=A^{(D)}_{(L)}
\end{align*}
and
\begin{align*}
  {(P_{L}A)}^{D}{(I_{n}-W_{1})}P_{L}
  &=U{\begin{bmatrix}{A_{L}}&B_{L}\\O&O\end{bmatrix}}^{D}
  \begin{bmatrix}{A_{L}}^{D}A_{L}&{A_{L}}^{D}B_{L}\\O&O \end{bmatrix}\begin{bmatrix}
    I_l & O \\
    O & O\end{bmatrix}U^*\\
&=U\begin{bmatrix}
  {A_{L}}^{D} & {({A_{L}}^{D})}^{2}B_{L} \\
  O & O
  \end{bmatrix}
\begin{bmatrix}{A_{L}}^{D}A_{L}&O\\O&O \end{bmatrix}U^*\\
&=U\begin{bmatrix}{A_{L}}^{D}&O\\O&O \end{bmatrix}U^*=A^{(D)}_{(L)}.
\end{align*}
Finally,
\begin{align*}
  {(I_{n}-W_{1})}A^{\dagger}{(I_{n}-W_{2})}
  &=U\begin{bmatrix}{A_{L}}^{D}A_{L}&{A_{L}}^{D}B_{L}\\O&O \end{bmatrix}
  {\begin{bmatrix} 
    A_L & B_L \\
    C_L& E_L
    \end{bmatrix}}^{\dagger}\begin{bmatrix}A_{L}{A_{L}}^{D}&O\\C_{L}{A_{L}}^{D}&O \end{bmatrix}U^*\\
&=U\begin{bmatrix}{A_{L}}^{D}&O\\O&O \end{bmatrix}
\begin{bmatrix} 
  A_L & B_L \\
  C_L& E_L
  \end{bmatrix}{\begin{bmatrix} 
    A_L & B_L \\
    C_L& E_L
    \end{bmatrix}}^{\dagger}\begin{bmatrix} 
      A_L & B_L \\
      C_L& E_L
      \end{bmatrix}\begin{bmatrix}{A_{L}}^{D}&O\\O&O \end{bmatrix}U^*\\
&=U\begin{bmatrix}{A_{L}}^{D}&O\\O&O \end{bmatrix}
\begin{bmatrix} 
  A_L & B_L \\
  C_L& E_L
  \end{bmatrix}
\begin{bmatrix}{A_{L}}^{D}&O\\O&O \end{bmatrix}U^*\\
&=U\begin{bmatrix}{A_{L}}^{D}&O\\O&O \end{bmatrix}U^*=A^{(D)}_{(L)}.
\end{align*} 
This completes the proof.
\end{proof}

\noindent In the next theorem, we use $A\mid_{L}$ to indicate the restriction of $A\in\mathbb{C}^{m\times{n}}$ to $L\leq\mathbb{C}^{n}$, that is, 
$A\mid_{L}x=Ax~~(x\in{L})$.
\begin{theorem}\label{TH3A}
  Let $A\in\mathbb{C}^{n\times{n}}$ and $L\leq\mathbb{C}^{n}$ be such that $\textup{Ind}(P_{L}AP_{L})=k$. Then
  \begin{equation}\label{ML}
  A_{(L)}^{(D)}=\tilde{A}^{-1}{(P_{L}AP_{L})}^{k},
  \end{equation}
where $S=\mathcal{R}{({(P_{L}AP_{L})}^{k})}$ and $\tilde{A}={(P_{L}AP_{L})}^{k+1}\mid_S$.
\end{theorem}

\begin{proof}
  The hypothesis $\tilde{A}={(P_{L}AP_{L})}^{k+1}\mid_S$ gives
  $\tilde{A}x=0$ where $x\in{S}$. There exists a $y\in\mathbb{C}^{n}$
such that $x={(P_{L}AP_{L})}^{k}y$. Moreover,
\[{(P_{L}AP_{L})}^{2k+1} y=\tilde{A} x=0.\]
Hence, we have $y\in\mathcal{N}({(P_{L}AP_{L})}^{2 k+1})=T$
which means $x=(P_{L}AP_{L})^{k} y=0$.
On the other hand, for every $y\in{S}$ there exists an $x\in\mathbb{C}^{n}$ such that
$y=(P_{L}AP_{L})^{k+1}((P_{L}AP_{L})^{k} x) \in \tilde{A}S$
since $S=\mathcal{R}({(P_{L}AP_{L})}^{2 k+1})$. 
Therefore $\tilde{A}$ is reversible.
\par We let $X=\tilde{A}^{-1}(P_{L}AP_{L})^{k}$ and consider the decomposition of any
$z\in\mathbb{C}^{n}$ as $z=z_{1}+z_{2}$ with $z_{1}\in T$ and
$z_{2}\in S$. It follows that there exists a
$t\in\mathbb{C}^{n}$ such that $z_{2}=(P_{L}AP_{L})^{k+1}t$ since $S=\mathcal{R}((P_{L}AP_{L})^{k+1})$.
Finally, we can verify that  $X=\tilde{A}^{-1}{(P_{L}AP_{L})}^{k}$ satisfies the three equations in $(\ref{dra})$.
\par The above-mentioned fact is true for the arbitrary $z \in \mathbb{C}^{n}$; thus we have
\[{(P_{L}AP_{L})}^{k+1} X={(P_{L}AP_{L})}^{k}, \quad X {(P_{L}AP_{L})} X=X, \quad  {(P_{L}AP_{L})} X=X {(P_{L}AP_{L})}.\]
It follows that $X={(P_{L}AP_{L})}^{D}=A_{(L)}^{(D)}$.
Now, Theorem {\ref{TH3A}} completes the proof.
\end{proof}

\section{Applications}

First off, we  discuss   applications of the  BDD-inverse in solving the restricted linear equations  in electrical network theory.

\begin{theorem}

Let $A\in\mathbb{C}^{n\times{n}}$, $L\leq\mathbb{C}^{n}$ and $k=\textup{Ind}{(AP_{L}+P_{L^{\perp}})}$. The restricted linear equations
\begin{equation}\label{gaoe1}
Ax+y=\beta ,~~x\in{L},~y\in{L^{\perp}},
\end{equation}
where $\beta\in\mathcal{R}{({(AP_{L}+P_{L^{\perp}})}^{k})}$.\\
\\
${(a)}$. The general solution of \eqref{gaoe1} is
  \begin{align*}
    &x=A_{(L)}^{(D)}\beta+{(P_L{A}P_{L})}^{k-1}(I-A_{(L)}^{(D)}(AP_{L}+P_{L^{\perp}}))u,\\
    &y=(I_{n}-AA_{(L)}^{(D)})\beta-{({A}P_{L})}^{k}(I-A_{(L)}^{(D)}(AP_{L}+P_{L^{\perp}}))u,
  \end{align*}
  where $u\in\mathbb{C}^{n}$ is arbitrary.\\
  ${(b)}$. If $(x+y)\in\mathcal{R}({(AP_{L}+P_{L^{\perp}})}^{k})$, then the unique solution of \eqref{gaoe1} is
  \begin{equation*}
    x= A_{(L)}^{(D)}\beta
    \text{ and } y = (I_{n}-AA_{(L)}^{(D)})\beta.
  \end{equation*}
\end{theorem}
\begin{proof}
{${(a)}$. Suppose that $z=x+y$, where $x=P_{L}z$ and $y=P_{L^{\perp}}z$.
The consistency of {(\ref{gaoe1})} is equivalent to the
consistency of
\begin{equation}\label{APLPlzZeq}
  (AP_{L}+P_{L^{\perp}})z=\beta.
\end{equation}
Since $k=\textup{Ind}{(AP_{L}+P_{L^{\perp}})}$ and $\beta\in{\cal R}(AP_{L}+P_{L^{\perp}})$, by {\cite[Lemma 2.4]{wei1998index}}
we know that {(\ref{APLPlzZeq})} has solutions, and its general solution is
\[z={(AP_{L}+P_{L^{\perp}})}^{D}\beta+{(AP_{L}+P_{L^{\perp}})}^{k-1}{(I-{(AP_{L}+P_{L^{\perp}})}^{D}{(AP_{L}
+P_{L^{\perp}})})}u,\]
 where $u\in\mathbb{C}^{n}$ is arbitrary.
Therefore,  the general solution of   {(\ref{gaoe1})} is
\begin{align*}
x&=P_{L}z=A_{(L)}^{(D)}\beta+{(P_L{A}P_{L})}^{k-1}(I-A_{(L)}^{(D)}(AP_{L}+P_{L^{\perp}}))u,\\
y&=\beta-Ax=(I_{n}-AA_{(L)}^{(D)})\beta-{({A}P_{L})}^{k}(I-A_{(L)}^{(D)}(AP_{L}+P_{L^{\perp}}))u.
\end{align*}}
${(b)}$.  Let $z=x+y\in\mathcal{R}({(AP_{L}+P_{L^{\perp}})}^{k})$ be written as $z={(AP_{L}+P_{L^{\perp}})}^{k}z_{1}$ for some $z_{1}\in \mathbb{C}^{n}$.
Then ${(AP_{L}+P_{L^{\perp}})}^{k+1}z_{1}=\beta$ and
\begin{align*}
z&={(AP_{L}+P_{L^{\perp}})}^{k}z_{1}\\
&={(AP_{L}+P_{L^{\perp}})}^{k+1}{(AP_{L}+P_{L^{\perp}})}^{D}z_{1}\\
&={(AP_{L}+P_{L^{\perp}})}^{D}{(AP_{L}+P_{L^{\perp}})}^{k+1}z_{1}\\
&={(AP_{L}+P_{L^{\perp}})}^{D}\beta.
\end{align*}
Since $\mathcal{R}({(AP_{L}+P_{L^{\perp}})}^{k})\cap\mathcal{N}({(AP_{L}+P_{L^{\perp}})}^{k})=\{0 \} $,
we know that $z={(AP_{L}+P_{L^{\perp}})}^{D}\beta$ is the unique solution of {(\ref{APLPlzZeq})}.
In consequence, the unique solution to {(\ref{gaoe1})} is
\begin{equation*}\label{zhang2}
  x=P_{L}z=A_{(L)}^{(D)}\beta,~y=\beta-Ax=(I_{n}-AA_{(L)}^{(D)})\beta.
\end{equation*}
\end{proof}

\noindent In the following, we consider a  constrained linear equation
  \begin{equation}\label{pax}
P_{L}Ax=b.
  \end{equation} 
where $A\in\mathbb{C}^{n\times{n}}$, $L\leq\mathbb{C}^{n}$, $k=\textup{ind}{(P_{L}AP_{L})}$ and $b\in\mathcal{R}{({(P_{L}AP_{L})}^{k})}$.
From \cite{wei1998index} and Theorem $\ref{T2}$ $(h)$, the general solution of $(\ref{pax})$ is given by
\begin{equation}\label{xb}
x=A_{(L)}^{(D)}b+{(P_{L}AP_{L})}^{k-1}{(I-A_{(L)}^{(D)}AP_{L})}z,~{z}~\text{an arbitrary vector.}
\end{equation}

\begin{theorem}\label{axb}
Let $A\in\mathbb{C}^{n\times{n}}$, $L\leq\mathbb{C}^{n}$. Also, suppose that $k=\textup{ind}{(P_{L}AP_{L})}$ and $b\in\mathcal{R}{({(P_{L}AP_{L})}^{k})}$.
The minimal P-norm solution of $(\ref{pax})$ is presented by
\begin{equation}\label{adb}
x_{opt}=A_{(L)}^{(D)}b
\end{equation}
where P is nonsingular matrix such that $P^{-1}{(P_{L}AP_{L})}P$ is the Jordan canonical form of $P_{L}AP_{L}$ and
$\left\|{x}\right\|_P=\left\|{P^{-1}x}\right\|_2$.
\end{theorem}

\begin{proof}
  Since
\begin{align*}
  \left\|{x}\right\|_{P}^{2}=&\left\|{P^{-1}\left[A_{(L)}^{(D)}b+{{(P_{L}AP_{L})}^{k-1}}{(I_{n}-A_{(L)}^{(D)}AP_{L})}z\right]}\right\|_{2}^{2}\\
=&{\left[P^{-1} A_{(L)}^{(D)} P P^{-1} b+P^{-1} {(P_{L}AP_{L})}^{k-1}\left(I-A_{(L)}^{(D)} AP_{L}\right) P P^{-1} z\right]^{\mathrm{T}}\left[P^{-1} A_{(L)}^{(D)} P P^{-1} b\right.} \\
& +\left.P^{-1} {(P_{L}AP_{L})}^{k-1}\left(I-A_{(L)}^{(D)} AP_{L}\right) P P^{-1} z\right]\\
= & \left(P^{-1} A_{(L)}^{(D)} P P^{-1} b\right)^{\mathrm{T}}\left(P^{-1}A_{(L)}^{(D)}  P P^{-1} b\right)+\left[P^{-1} {(P_{L}AP_{L})}^{k-1}\left(I-A_{(L)}^{(D)} AP_{L}\right) P P^{-1} z\right]^{\mathrm{T}}\\
&\times P^{-1} {(P_{L}AP_{L})}^{k-1}\left(I-A_{(L)}^{(D)} AP_{L}\right) P P^{-1} z+\left(P^{-1} b\right)^{\mathrm{T}}\\
&\times{\left\{\left(P^{-1} A_{(L)}^{(D)} P\right)^{\mathrm{T}}\left[P^{-1} {(P_{L}AP_{L})}^{k-1}\left(I- A_{(L)}^{(D)}AP_{L}\right) P\right]\right\}} 
P^{-1} z+\left(P^{-1} z\right)^{\mathrm{T}}\\
&\times\left\{\left[P^{-1} {(P_{L}AP_{L})}^{k-1}\left(I-A_{(L)}^{(D)}AP_{L}\right) P\right]^{\mathrm{T}}\left(P^{-1} A_{(L)}^{(D)} P\right)\right\} P^{-1} b\\
\geq&\left\|P^{-1} A_{(L)}^{(D)} P P^{-1} b\right\|_{2}^{2}+\left\|P^{-1} {(P_{L}AP_{L})}^{k-1}\left(I-A_{(L)}^{(D)}AP_{L}\right) P P^{-1} z\right\|_{2}^{2}\\
=&\left\|A_{(L)}^{(D)} b\right\|_{P}^{2}+\left\|{(P_{L}AP_{L})}^{k-1}\left(I-A_{(L)}^{(D)}AP_{L}\right) z\right\|_{P}^{2} \geqslant\left\|A_{(L)}^{(D)} b\right\|_{P}^{2}.
\end{align*}
The proof is complete.
\end{proof}

\section{Conclusion}
In this work. Our main goal is to present a new generalized inverse, that is, BDD-inverse,
which extends the notion of the Bott-Duffin inverse.
Note that the inverse exists for any square matrix.
We show some of its properties, characterizations and representations.
Furthermore, the BDD-inverse can be used in solving appropriate systems of constrained matrix equations.
The minimum P-norm solution of the corresponding constraint equation is given by applying the BDD-inverse.
Finally, we discussed Cramer's rule for the minimum P-norm solution of constraint equation.
However, based on current research, there are many topics that can be discussed in the BDD-inverse. Some ideas for further research are given below:
\vspace{0.2cm}
\par${(a)}$ The relationship between BDD-inverse and other generalized inverses.
\par${(b)}$ Continuity and perturbation of the BDG-inverse, and the use of iterative methods to calculate the BDG-inverse.
\section*{Acknowledgments}
\noindent The authors would like to
thank the Referees for their valuable comments and suggestions which helped
us to considerably improve the presentation of the paper.
\section*{Disclosure statement}
\noindent No potential conflict of interest was reported by the authors.

\section*{Funding}
\noindent This work is supported by the National Natural Science Foundation of China [grant number
11961076].

\bibliographystyle{tfnlm}
\bibliography{interactnlmsample}
\end{document}